\documentclass[letterpaper]{amsart}
\usepackage{ifpdf}
\pdfoutput=1
\usepackage[utf8]{inputenc}
\usepackage[T1]{fontenc}
\usepackage{lmodern}
\usepackage{amssymb}
\usepackage{mathtools}
\theoremstyle{plain}
\newtheorem{theorem}{Theorem}[section]
\newtheorem{proposition}[theorem]{Proposition}
\newtheorem{lemma}[theorem]{Lemma}
\newtheorem{corollary}[theorem]{Corollary}

\theoremstyle{definition}

\newtheorem{example}[theorem]{Example}
\newtheorem{remark}[theorem]{Remark}

\theoremstyle{remark}
\newtheorem*{acknowledgements}{Acknowledgements}

\numberwithin{equation}{section}
\usepackage{pgf}
\usepackage{tikz-cd}
\usepackage[scaled]{helvet}
\def\Comment#1{{\rm\textbf{\textsf{[#1]}}}}
\def\comment#1{\Comment{#1}}
\def\comment#1{}

\def\doi#1{\href{http://dx.doi.org/#1}{doi:#1}}

\def\HG{H_{G}}

\def\F{\mathbb F}
\def\Z{\mathbb Z}
\def\R{\mathbb R}
\def\k{\it k}

\def\Q{\mathbb Q}

\begin{document}

\title {Equivariant cohomology of $\mathbf{(\Z_{2})^{r}}$-manifolds and syzygies}

\author{Volker Puppe}
\address{Fachbereich Mathematik, Universität Konstanz, 78457 Konstanz, Germany}
\email{volker.puppe@uni-konstanz.de}

% aspell:off
%\hypersetup{pdfauthor=\authors}
% aspell:on

\subjclass[2010]{Primary 57R91; secondary 13D02, 57S25, 55M35}

\begin{abstract}
We consider closed manifolds with $(\Z_2 )^r $-action, which are obtained as intersections of products of spheres of a fixed dimension with certain `generic' hyperplanes. This class contains the real versions of the
`big polygon spaces' defined and considered by M.Franz in \cite{Franz:2014}. We calculate the equivariant cohomology with $\F_2$-coefficients, which in many examples turns out to be torsion-free but not free and realizes all orders of syzygies,
which are in concordance with the restrictions proved in \cite{AlldayFranzPuppe:manu}. ~The final results for the real versions are analogous to those for the big polynomial spaces in  \cite{Franz:2014}, where $(S^1)^r$-actions and rational coefficients are considered, but we consider also  a wider class of manifolds here and the point of view as well as the method of proof, for which it is essential to consider equivariant cohomology for divers - but related - groups, are quite different.

\end{abstract}

\maketitle

\section{Introduction}

In the papers~\cite{AlldayFranzPuppe:orbits1} and ~\cite{AlldayFranzPuppe:orbits4}
the equivariant cohomology (with coefficients in a field of characteristic $0$) of spaces equipped with an action of a torus~$T=(S^{1})^{r}$ was studied, 
in particular the relation between the so-called Atiyah--Bredon sequence
and the notion of syzygies coming from commutative algebra.
Among the results is the following theorem (see ~\cite{AlldayFranzPuppe:orbits1}, Cor.~1.4):
\begin{theorem}
Let $X$ be a compact orientable $T$-manifold. If $H^*_T(X)$ is a syzygy of order $k \geq  r/2$, then it is free over $H^*(BT)$.
\end{theorem}
In~\cite{Franz:2014} examples of $T=(S^{1})^{r}$-manifolds
were given, which show that the restriction on the order of syzygies obtained in ~\cite{AlldayFranzPuppe:orbits1},
are sharp. 
Coefficients were taken in $\Q$.
% a field of characteristic~$0$.

In this note we consider actions of a $2$-torus $G=(\Z_2)^{r}$ and coefficients in the field $\F_2$ of characteristic $2$.
All major analogous results of~\cite{AlldayFranzPuppe:orbits1} and~\cite{AlldayFranzPuppe:orbits4}, in particular Theorem~ 1.1 above,
turn out to be true in this setting. Nevertheless, some of them require new methods of proof, basically because, 
in contrast to $T=(S^{1})^{r}$, 
$G=(\Z_2)^{r}$ has only finitely many
subgroups and because the field~$\F_2$ has only finitely many elements.
This is carried out in a so far unpublished manuscript ~\cite{AlldayFranzPuppe:manu}, even for the not quite analogous case of 
$G=(\Z_{p})^{r}$-actions and  coefficients in a field $\k$ of characteristic~$p>0$, $p$ an odd prime, 
which is somewhat more involved than the $p=2$ case.

The results in this note for the real versions of the 'big polygon spaces' with $(\Z_2)^r$-actions  are in a 
sense analogous to those for $(S^1)^r$-actions in ~\cite{Franz:2014},
i.e. among other results we show (see Cor.~3.16, Remark ~3.17, and compare ~\cite{Franz:2014} , Cor.~5.3 for the case of $(S^1)^r$-actions):
\begin{theorem}
Let $k$ and $r$ be integers with $k < r/2$, then there exists a compact $(\Z_2)^r$-manifold, $N_0$, such that $H^*_{(\Z_2)^r} (N_0;\F_2)$
is a $k$-th syzygy over $H^*(B(\Z_2)^r;\F_2)$ but not a $(k+1)$-th syzygy.
\end{theorem}
Compared to ~\cite{Franz:2014} we take a different point of view and the proofs are also quite different.
We consider the equivariant cohomology for divers groups acting on a class of manifolds, which contains the real analogues of the `big polygon spaces`of M.Franz, but also the more general 'big chain spaces` (see ~Remark \ref{spaces}, (3) and (4), Remark~3.10 and Cor.~3.13) which are not considered in  ~\cite{Franz:2014}. \\
Certain familiarity with equivariant cohomology and P.A.~Smith-Theory is assumed throughout.
Standard references are e.g. ~\cite{Borel},~\cite{Bredon},~\cite{AlldayPuppe:1993}.
\begin{acknowledgements}
This note is based on joint work with C.Allday and M.Franz and numerous discussions with M.Franz. 
Some of the methods applied go back to a conversation with J.-C.Hausmann and M.Farber in 2007 about 
equivariant aspects of the Walker conjecture. I also want to thank Matthias Franz
for helpful comments and support to create an acceptable LateX file.
\end{acknowledgements}
\section{Some basic definitions and fundamental results}
We will not cite here all analogous results to those in ~\cite{AlldayFranzPuppe:orbits1} and ~\cite{AlldayFranzPuppe:orbits4}
for the case of 2-tori and $\F_2$-coefficients
but concentrate on the result which is relevant in view of the later examples.

Let $G=(\Z_2)^{r}$ be a $2$-torus and
%\begin{equation}
 $ R = H^{*}(BG) =\F_2[x_{1},\dots,x_{r}]$
%\end{equation}
the polynomial ring in the variables $x_{1},\dots,x_{r}$ of degree~$1$.

We recall the notion of syzygies from commutative algebra, see e.g.~\cite{BrunsVetter:1988}. 
A finitely generated $R$-module~$M$ is called a \emph{$j$-th syzygy}
if there is an exact sequence
\begin{equation}
0\to M\to F^{1}\to \dots \to F^{j}
\end{equation}
with finitely generated free $R$-modules~$F^{1}$,~\ldots,~$F^{j}$.
The first syzygies are exactly the torsion-free $R$-modules,
and the $j$-th syzygies with~$j\ge r$ are the free modules.

An easy way
to obtain syzygies over the polynomial ring $R$ is to use the Koszul resolution
\begin{equation}
\label{eq:koszul-resolution} 0 \longrightarrow
  R \stackrel{\delta_{r}}\longrightarrow
  R^{\binom{r}{r-1}} \longrightarrow
\cdots \longrightarrow R^{\binom{r}{1}} \stackrel{\delta_{1}}\longrightarrow R \stackrel{\delta_{0}}\longrightarrow \F_2  \longrightarrow 0 ~, %\;;
\end{equation}
indeed, the image of~$\delta_{j}$, $K_j$, is obviously a $j$-th syzygy by definition, but it is not a $(j+1)$-th syzygy, because the homological dimension over $R$, $hdim_R(M)$, of a $(j+1)$-th syzygy $M$ over $R$ is at most $r-(j+1)$, while $hdim_R(K_j) = r-j$.

In ~\cite[Prop.]{Allday:1985} Allday proves for rational coefficients the following result
for a suitable Poincaré duality ($PD-$ for short) space~$X$ on which $T = (S^1)^r$ acts,
e.g. a compact orientable T-manifold.
If $H_T^{*}(X)$ has homological dimension $1$,
then it has $H^*(BT)$-torsion.
In particular, if~$r=2$, i.e., if $T=S^{1}\times S^{1}$,
then $H^*_T (X)$ is torsion-free if and only if it is free.\\
Analogous results hold for $(\Z_2)^r$  instead of $T$,
and $\F_2$-coefficients.
But the above equivalence breaks down for~$r>2$;
see ~\cite{FranzPuppe:2008} for counterexamples.
The correct generalization of Allday's result is as follows.

\begin{proposition}
\label{thm:torsion-free-hd-1} Let X be a PD-space with a $G$-action, which is a compact $G-CW$ complex, e.g.
a compact orientable G-manifold.\\
If $\HG^{*}(X)$ is a $k$-th syzygy for some $k~\geq r/2 $,
then it is free over $H^*(BG)$.

\end{proposition}

\begin{proof}
Compare ~\cite{AlldayFranzPuppe:orbits1}, Proposition 5.12(2) for the result in case of $G = (S^1)^r$-actions and rational coefficients.
A proof for the case $G = (\Z_2)^r$ and $\F_2$-coefficients is contained in ~\cite{AlldayFranzPuppe:manu}. 
\end{proof}

While in ~\cite{AlldayFranzPuppe:orbits1}
it was shown by examples that
actually all orders of syzygies can occur as equivariant cohomology modules of non-compact $G$-manifolds; examples of compact $G$-manifolds or $PD$-spaces which realize all orders of syzygies $< r/2$  were not given there. This is done
in ~\cite{Franz:2014} for the case of  $(S^1)^r$-actions. Here we give similar results, but essentially different proofs for $G$-actions, where  $G=(\Z_{2})^{r}$.
\\
From now on we always take $\F_2$-coefficients.
A $G$-space is \emph{equivariantly  formal} in the sense of ~\cite{GKM} if and only if the equivariant cohomology, $H^*_G(X)$ is isomorphic to $H^*(X) \otimes R$ as an $R$-module (but not necessarily as an $R$-algebra). 
We denote this property be  $\it{CEF}$  (\emph{cohomologically equivariantly formal})  to distinguish it from notions of formality in rational homotopy theory.\\

The following Mayer-Vietoris type theorem is basic for our calculations.

\begin{theorem}
\label{MV}
Let $M$ be a $G$-manifold, $G$ a $2$-torus. Assume that $M^+,M^- \subset M$ are $G$-invariant,  with
$M=M^+ \cup M^-$, such 
that $M,M^+,M^-$ are $\it {CEF}$, and the action on $M^0:=M^+\cap M^-$ is fixed point free.\\ 
Assume also that the maps 
$H^*(M)\longrightarrow H^*(M^{\pm})$ induced by the inclusions are surjective. 
Then one has the following Mayer-Vietoris diagram:

\begin{equation}
\begin{tikzcd}
H^*_{G}(M) \arrow{r} 
	  \arrow{d} 
	& H^*_{G}(M^+) \arrow {d}\\
    H^*_{G}(M^-) \arrow {r}  
 & H^*_{G}(M^0)
\end{tikzcd}
\end{equation}
\\
All maps in the above diagram are surjective and the long exact Mayer-Vietoris sequence decomposes into short exact sequences

\begin{equation}
0 \longrightarrow {H_G^{*}(M)} \stackrel {(\xi ^+ ,\xi^-)} \longrightarrow {H_G^{*}(M^+)} \oplus {H_G^{*}(M^-)} 
\longrightarrow {H_G^{*}(M^0)} \longrightarrow 0
\end{equation}
one has
\begin{equation}
{H_G^{*}(M^0)} \cong {H_G^{*}(M)}/{({ker \xi^+}\oplus {ker \xi^-})} \cong ({H_G^{*}(M^+)}\oplus {H_G^{*}(M^-)})/{(\xi^+,\xi^-){H_G^{*}(M)}}.
\end{equation}
\end{theorem}

\begin{proof} Since $M,M^\pm$ are $\it {CEF}$, one has ${H_G^{*}(M)} \cong  {H^{*}(M)}\otimes R $
and ${H_G^{*}(M^\pm)}\cong {H^{*}(M^\pm)}\otimes R$ as $R$-modules.
Although the maps in equivariant cohomology might not be the canonical extensions of the 
corresponding maps in non-equivariant cohomology, still the former are surjective because the latter are so 
by assumption (see e.g. ~\cite{AlldayPuppe:1993} Lemma (A.7.3)(2)) .
On the other hand, because the fixed point set $(M^0)^G$ is empty, one has
\begin{equation}
H^*_G(M^G) \cong H^*_G((M^+)^G) \oplus H^*_G((M^-)^G).
\end{equation}
Since the inclusions of the fixed point set 
of $M$ induces  an injection in equivariant cohomology, the map 
$(\xi^+,\xi^-): H^*_G(M) \longrightarrow H^*_G(M^+) \oplus H^*_G(M^-)$
is also injective, for the composition with 
$H^*_G(M^+) \oplus H^*_G(M^-) \longrightarrow 
H^*_G((M^+)^G)\oplus H^*_G((M^-)^G)$ \\ coincides with the injective map
\begin{equation} 
H^*_G(M) \longrightarrow H^*_G(M^G) \cong H^*_G((M^+)^G)\oplus H^*_G((M^-)^G).
\end{equation}
This means that the long exact Mayer-Vietoris sequence decomposes into short exact sequences and the above Mayer-Vietoris diagram is cocartesian. Hence together with the maps on the upper and left side of the diagram also those on the lower and right side are surjective.
\end{proof}

\begin{remark}
Under the conditions of Theorem ~\ref{MV}
the sequence (2.4) is a free resolution of the $R$-module $H^*_G(M^0)$ since  $H^*_G(M)$ and $H^*_G(M^{\pm})$ are free $R$-modules.
Since $\xi ^{\pm}$ are surjective, $ker \xi ^{\pm}$ are also free, and $ker \xi ^+ \cap ker \xi^- = 0$ 
since $(\xi ^+,\xi ^-)$ is injective.
Hence the following sequence is also a free resolution of $ H^*_G(M^0)$:
\begin{equation}
\begin{tikzcd}
0  \arrow{r} & ker \xi ^- \arrow{r} & {H_G^{*}(M)/ker ~\xi ^+} 
\arrow{r} & {H_G^{*}(M^0)} \arrow{r} & 0 
\end{tikzcd} 
\end{equation}
Here we identify $ker \xi ^-$ with its isomorphic image in $H^*_G(M)/ ker ~\xi ^+$.
\end{remark}
\section{The Manifolds}
The manifolds we consider in this section are intersections of products of spheres of a fixed dimension with 
a number of hyperplanes of a particular type. The actions are just the restrictions of the canonical action on the ambient Euclidean space to some of the coordinates. Among these manifolds are the real analogues of the `big polygon spaces' considered in 
\cite{Franz:2014} (cf.~Remark~\ref{spaces}.(1) below). We calculate the equivariant cohomology with respect to different groups. It turns out
that the equivariant cohomology  for these manifolds with respect to certain subgroups is often torsion-free but not free and realizes all orders of syzygies which are in concordance with Proposition~\ref{thm:torsion-free-hd-1}.\\
Let $S^{m+n-1}:=\{(x_1,...,x_m,y_1,...,y_n) \in {\R}^{m+n};~ \sum_{i=1}^m x_i^2 +\sum_{j=1}^n y_j^2 = 1\}$ ,
%and 
$M:=(S^{m+n-1})^r$ for $m\geq 2,n \geq 0, r\geq 1$.
A point $ w \in M$ is given by the coordinates \\
 $w=((x_{1,1},.,x_{m,1},y_{1,1},.,y_{n,1}),(x_{1,2},.,x_{m,2},y_{1,2},.,y_{n,2}),.,(x_{1,r},.,x_{m,r},y_{1,r},.,y_{n,r})).$\\
Let $\ell:=(l_1,...,l_r) \in {(\R\setminus \{0\})}^r$. The vector $\ell$ is called $generic$ if and only if 
$\sum _{j=1}^r l_j \epsilon _j \neq 0$ for all $\epsilon _j = \pm 1$. We assume throughout this note that $\ell$ is generic.
We define $f_i:M\rightarrow\R$ by $f_i(w):= \sum_{j=1}^{r} {l_{j}x_{i,j}}$ for $i=1,...,m.$
We actually assume that all $l_{j}$ are positive, since one can replace the coordinate $x_{i,j}$ by $-x_{i,j}$,
if necessary.
Set $M_0:= M$, and  for $i=1,...m$, $M_{i}:=\{w\in M;f_{\mu}(w)=0$ for $\mu =1,...i\}$;\
furthermore let $g_i:=f_i|_{M_{{i-1}}}$, so $g_i:M_{{i-1}}\rightarrow \R$ .
We put $N_c: = g_m^{-1} (c)$, for $c \in \R $. Since $N_c$ is homeomorphic to $N_{-c}$ (by multiplying the coordinates
$(x_{m,1},...,x_{m,r})$ with $-1$), we may assume $c \geq 0$.
 
\begin{remark}
 \label{spaces}
Spaces of the above type have been considered by many mathematicians in different contexts,
 see ~\cite{FHS}, ~\cite{FS}, ~\cite{FF}, ~\cite{Franz:2014},~\cite{Hausmann:book} and the references therein, e.g.:

\begin{enumerate}

\item{Big polygon spaces}: These spaces are studied in \cite{Franz:2014}. They are the complex version of the above spaces $N_0=M_m$ with 
$\ell=(l_1,...,l_r) \in {\R}^r$ and $n \geq 1$ considered with an action of $(S^1)^r$.
(It is shown in ~\cite{Franz:2014} that without loss of generality one may assume that $\ell \in {(\R \setminus \{0\})}^r.$)

\item{Polygon spaces}:
If $\ell =(l_1,...,l_r) \in {\R}^r_{>0}$, 
%and $n = 0$.
the space $\{w\in M_{ m}; ~y_{\nu,1}=...=y_{\nu,r}=0$ for $\nu=1,..,n \}$
(which amounts to the same as just taking $n = 0$) 
is homeomorphic to the space  of polygons (resp. the free polygon space)
$\tilde{\mathcal N}_{m}^r (\ell )$ in ${\R}^m$ for the length vector $\ell $ (cf.~\cite{Hausmann:book}, Chapt.10.3; p.445). These spaces (named $E_m(\ell)$ in ~\cite{FF}) and in particular their non-equivariant cohomology with $\Z_2$-coefficients are studied in
~\cite{FS} and ~\cite{FF}.

\item{Big chain space}:
With the notation as in (2), the space\\ $\{w\in f_m^{-1}(c);~y_{\nu,1}=...=y_{\nu,r}=0$ for $\nu =1,..,n
%;  y_1 = ...=y_r = 0
\}$ is homeomorphic 
to the big chain space
$\mathcal {BC}_m^{r+1}(\tilde \ell)$ for $\tilde \ell = (l_1,...,l_r,c)$, in  \cite{Hausmann:book}, Chapt.10.3; p.444. We also consider $N_c$ for $n \geq 1$ and $c \neq 0$ as a big chain space.

\item{Chain spaces}:
Again with the notation as in (2), the space\\ $\{w\in g_{m}^{-1}(c);
~y_{\nu ,1}=...= y_{\nu ,r}=0$ for $\nu =1,..,n
\}$ is homeomorphic to the 
chain space $\mathcal C ^{r+1}_m (\tilde \ell)$ for $\tilde \ell = (l_1,...,l_r,c)$, in  \cite{Hausmann:book}, Chapt.10.3; p.444.
They were also studied in ~\cite{FHS}.
\end{enumerate}
\end{remark}

Our aim is to calculate the equivariant cohomology of the spaces $M_{ i}$, which are just intersections of M with certain hyperplanes, with respect to the standard linear $G =(\Z_2)^r$-action on all coordinates $y_{\nu,1},...,y_{\nu,r}$ for $\nu=1,..,n$. So the spaces described in (2) and (4) of the above remark occur as the fixed point sets of these actions on the corresponding spaces, where $y_{\nu ,1},...,y_{\nu ,r}$ for $\nu =1,..,n$ are arbitrary.

\begin{proposition}

 \label{Morse}
(1) For $i=1,...,m$ the functions $f_i:M \rightarrow \R$ are Morse functions with the set of isolated critical
points $C_i:=\{w
\in M$; all coordinates equal to zero, except $x_{i,j} = \pm 1$ for $j=1,...,r\}.$\\
(2) For $i= 1,...,m$ one has: If $\ell$ is generic, then  $M_{i-1}$ is a closed manifold, and $g_i:M_{i-1} \rightarrow \R$
is a Morse function with isolated critical points $C_i$ as above.

\end{proposition}

\begin{proof}
The proof of part (1) is essentially contained in ~\cite{Hausmann:book}, Lemma 10.3.1.
We give some details and introduce notation in view of the proof of part (b).

Part (1): The map $S^{m+n-1}  \rightarrow \R$ given by $(x_{1,j},...,x_{m,j},y_{1,j},...,y_{n,j}) 
\mapsto l_{j}x_{i,j}$\\ for a single sphere  is clearly a Morse function with two non-degenerated critical points, given by $x_{i,j} = \pm 1$ and all other coordinates equal to $0$. The corresponding Hesse matrix is a diagonal matrix  with entries $\pm l_{i,j}$ along the diagonal, if one uses the coordinate system for the sphere around the critical points obtained by projecting to the coordinates
different from $x_{i,j}$. It follows that for the map $f_i$ on $(S^{m+n-1})^r$  the set of critical points is just $C_i$ above,
and the Hesse matrix with respect to the coordinate systems chosen above is a huge diagonal matrix, which after arranging the 
variables in lexicographical order (i.e.: $w=(x_{1,1},x_{1,2},...,x_{1,r},x_{2,1},...x_{2,r},...,x_{m,r},y_{1,1},...,y_{n,r})$, and
omitting the coordinates $x_{i,j}$ for $j= 1,...,r)$, can be view as consisting of $m+n-1$ diagonal blocks of the form
\begin{equation}
\left(
 \begin{array}{cccccc}
 \pm l_{1}  &  0 &  0 &   .  & .  & .  \\
0  & \pm l_{2} &  0  & .   &.  & .   \\
.\\
.\\
.\\
0 & 0 & 0  & .  & .    &\pm l_{r}

\end{array}
\right)
\end{equation}
\\
In particular, all critical points in $C_i$ are regular, so $f_i$ is a Morse function, 
and $(\R  \setminus  {f_i(C_i)})$ is the set of regular values.

Part (2): Proof by induction on $i$. The beginning, $i=1$, is clear by part~(a).\\
By induction hypothesis 
$g_{i-1}: M_{ {i-2}} \longrightarrow \R$ is a Morse function on the compact manifold $M_{{i-2}}$
 with isolated critical points $C_{i-1}$.
Since $\ell $ is generic, $g_{i-1}$ is non zero on $C_{i-1}$. So $0 \in \R$ is a regular value of $g_{i-1}$ and hence 
$g_{i-1}^{-1}(0) = M_{{i-1}}$
is a compact manifold.
 Since $C_i$ above is contained in $M_{i-1}$, the points in $C_i$ are also critical points of $g_i$. We first show that there are no other critical points of $g_i$.
For a critical point $w \in M_{i-1}$ of $g_i$ the differential $dg_i$ must be normal to  $M_{i-1}$. This gives the following conditions for the single coordinates :
%\begin{equation}
\begin{enumerate}
\item
 $l_j = c_j x_{i,j}$ for $j=1,...,r$ and some $c_j \neq 0$
\item
 $x_{\mu ,j} = 0$ for $ \mu >i; j=1,...,r$, and $y_{\nu,j} = 0 $ for $\nu = 1,...,n; j=1,...,r$
\item 
$x_{\mu ,j} /x_{i,j} = c_{\mu}$ for $\mu < i; j= 1,...,r$
\end{enumerate}
%\end{equation}
On the other hand one has 
\begin{equation}
\sum _{\mu =1} ^m x^2_{\mu,j} + \sum_{\nu =1} ^n y_{\nu ,j} = \sum _{\mu =1} ^i x^2_{\mu,j} = x^2_{i,j} +
\sum _{\mu =1} ^{i-1} c_\mu x^2_{i,j} = 1
\end{equation}
Hence $x^2_{i,j} = (1 + \sum _{\mu = 1} ^{i-1} c^2_\mu )^{-1}$ is independent of j.
Similarly for $\mu <i$ one has $x^2_{\mu ,j} = 1 - \sum _ {\xi \neq \mu} x^2 _{\mu ,j} = 
1 - \sum _ {\xi \neq \mu}  c^2 _\mu x^2 _{i,j}$ is also independent of $j$.
Since $l$ is generic $\sum _{j=1} ^r l_j x_{\mu ,j} = 0$ for $ \mu = 1,..., i-1$ implies
that $x_{\mu ,j} = 0 $ for $\mu = 1,...i-1; j= 1,...,r.$ \\
All together one has $x_{i,j} ^2 = 1$ for $j=1,...,r$ and $x_\mu ,j = 0$ for $j=1,...,r$ and $\mu \neq i$.
Also $y_{\nu ,j} = 0 $ for $\nu = 1,...,n; j = 1,...,r.$
This means that the set of singular points of $g_i$ is $C_i$.

We show next that the critical points are regular. This amounts to proving that the diagonal Hesse form of part (a) is still regular when restricted to the intersection with the linear subspaces of $(\R^{m+n})^r$ given by the 
equations $\sum_{j=1}^r l_{j}~x_{\mu ,j} = 0$ for $\mu=1,...,i-1$
Without restriction we may assume that $l_{r} = 1$. 
Hence for  $\mu =1,...,i-1$ we have $x_{\mu ,r} = - \sum_{j=1}^{r-1} l_{j}~x_{\mu ,j}$.
For the coordinates\\
$(x_{1,1} \dots x_{1,r-1}) \dots (x_{i-1,1}\dots x_{i-1,r-1}),(x_{i,1}\dots x_{i,r})\dots (x_{m,1}\dots x_{m,r}),(y_{1,1} \dots y_{n,r})$\\
we get the same blocks for the Hesse matrix as in part~(a), except for those where the coordinate $x_{\mu ,r}$ is deleted.
In the latter case the blocks look as follows:

\begin{equation}
\left(
 \begin{array}{cccccc}
{ \pm l_{1} +  l_{1}^2}  &  l_{1}~ l_{2}  &   .  &  . &  . &     l_{1}~ l_{r-1} \\
 l_{1} l_{2}  & \pm l_{2} +   l_{2}^2 &  .  &   . &   .&   l_{2} l_{r-1} \\
.\\
.\\
.\\
 l_{1}~ l_{r-1} & . & .  &  . &  .& \pm l_{r-1} +  l_{r-1} ^2

\end{array}
\right)
\end{equation}

We have to show that the determinant of such a block matrix is non-zero.  Adding appropriate multiples of the last column to the first (r-2) columns gives the following matrix:
\begin{equation}
\left(
 \begin{array}{cccccc}
{\pm l}_{1}  &   0   &   .  &  .  &   .&    l_{1}~ l_{r-1} \\
 0  & \pm l_{2}  &  .  &  .  &   .&    l_{2}~ l_{r-1} \\
.\\
.\\
.\\
\mp   l_{1}~ l_{r-1} & \mp   l_{2}~ l_{r-1} & .  &  . &  .& \pm l_{r-1} +  l_{r-1} ^2

\end{array}
\right)
\end{equation}

Adding $(\pm 1)$ times the first (r-2) rows to the last row gives a triangular matrix with determinant
$(\mp   l_{1})(\mp   l_{2})...(\mp   l_{r-2})( l_{r-1})(\pm 1  \mp   l_{1} \mp   l_{2}...\mp   l_{r-2} \pm l_{r-1}).$\\
This term is non-zero since  $ l_{j} \not = 0$  for $j=1,...,r$ as is $(\pm 1  \mp   l_{1} \mp   l_{2}...\mp   l_{r-2} \pm l_{r-1})$, because $l$ is generic.
\end{proof}

The Morse inequalities give the following result:

\begin{corollary}
 \label{inequality}
For $i=1,...,m$,   ~$dim_{\F_2} H^*(M_{i-1}) \leq  2^r.$
\end{corollary}
 
The set $C_i$ can also be viewed as the fixed point set of the involution on $M_{i-1}$, which is given by 
multiplying all coordinates of $m\in M_{i-1}$ with $\pm 1$, except for the  $x_{i,j}, j=1,...,r$.
By the P.A.~Smith inequalities one has 
 $dim_{\F_2}   H^*(M_{i-1})\geq 2^r$.

\begin{corollary}
\label{perfect}

The maps $g_i:M_{i-1} \rightarrow \R$ are perfect Morse functions for $i=1,...,m$ and coefficients $\F_2$, and  $dim_{\F_2} H^*(M_{i-1})= 2^r$.

\end{corollary}

If a 2-torus acts on $M_{i-1}$ with fixed point set $C_i$, then this action is $\it {CEF}$. 
Furthermore: If $g_i$ is equivariant with respect to the 2-torus action and the trivial action on $\R$ then for a regular value $c \in \R$ of $g_i$ the subspaces
 $g^{-1}(-\infty,c) \simeq g^{-1}(-\infty,c]$ and $g^{-1}(c,\infty) \simeq g^{-1}[c,\infty)$ are also $\it {CEF}$.
% and $g^{-1}(c)$ does not contain fixed points 
If the equivariant cohomology of these spaces is known for a regular value $c \in \R$ of $g_i$,
then one can calculate the equivariant cohomology of $g^{-1}(c)$ by a Mayer-Vietoris argument (see Theorem ~\ref{MV}).
We apply this below. 
%In the following we give several examples for this method.
Let the action of  $\tilde G_i = \Z_2 \times (\Z_2)^r $ on $M$ be defined by the standard linear presentation of $(\Z_2)^r$ on the coordinates $y_{\nu,1} ,...,y_{\nu,r}$ and the `diagonal' involution which multiplies the coordinates $x_{\mu,\nu}$ by $\pm 1$, except for the coordinates $x_{i,j}, j=1,...,r$ where the action is trivial. Note that $C_i = M^{\tilde G _i}$ is contained in $M_{{i-1}}$, 
and  $M_{{i-1}}$ is $\tilde G_i$-invariant. So the $\tilde G_i$-action on $M_{i-1}$ is $\it {CEF}$ for $i=1,...,m$.\\
Although we are mainly interested in calculating the equivariant cohomology of the manifolds 
$M_{ i}$ with respect to the action of $G$ above, it turns out to be useful to calculate the equivariant cohomology with respect to the action of the bigger group $\tilde G_i $ first.

\begin{proposition}
 \label{HM}

$H^*_{\tilde {G_i}}(M) \cong \F_2[s_1,...,s_r ,t,t_1,...,t_r]/\{s_j \overline s_j,~ j=1,...,r\}$,\\
where $ \overline s_j =s_j + t^{m-1}t_j ^{n}, (|s_j| = m+n-1, |t_j| = |t|=1)$.

\end{proposition}

\begin{proof}
The proof follows by induction on $r$ using the Künneth Theorem (cf.\cite{Hausmann:book}, Proposition 10.3.5, with different notation).
\end{proof}

For $J = \{j_1,...,j_k\} \subset \{ 1,...,r\}$ we denote by $p_{i,J}$ the fixed point in $C_i \subset M$
with\\
$x_{i,j} = \begin{cases}
+1  & \text{for $j \in J$} \\-1 & \text{for $j\not \in J$}
\end{cases}$
\\
%\noindent
In Proposition \ref{HM} we can choose the variables $s_1,...,s_r$ in such a way that the restriction to a fixed point $p_{i,J}$ in equivariant cohomology is given by the following proposition.
\begin{proposition}
\label{FP}
The inclusion $p_{i,J}  \in  M$ induces the following map in equivariant cohomology\\
\mbox
{$H^*_{\tilde G_i}(M) \cong\F_2[s_1,...,s_r,t,t_1,...t_r]/\{s_j \overline s_j,~ j=1,...,r\} \longrightarrow
 H^*_{\tilde G_i}(p_{i,J}) \cong\F_2[t,t_1,...,t_r] $,}\\
$s_j \longmapsto  \left \{
 \begin{array}{ccc}
t^{m-1} t_j ^n
 &  $ if $  & j \in J \\
0  &  $ if $  & j \not \in J \\
\end{array}\right.$
\end{proposition}
 
\begin{proof}
Again the proof can be given by induction on $r$ using the Künneth Theorem.
\end{proof}
 Our main aim is to calculate the equivariant cohomology of $N_0=M_{ m}$.
Before we get to the equivariant cohomology of $N_0$ with respect to 
the $G ={(\Z_2})^r$-action we consider the equivariant cohomology with respect to 
an extended action of the group $\tilde G = \Z_2 \times G$ where the first factor $\Z_2$
acts by multiplication with $\pm 1$ on all coordinates $x_{i,j}$ except $x_{m,1},...,x_{m,r}$. 
This coincides with the action of $\tilde G _m$ above.
The maps $g_i$ are equivariant with respect to the above 
action, if one takes the trivial action on $\R$. The fixed point sets $M^{\tilde G}, M_{ {m-1}}^{\tilde G}$
coincide and are equal to the above defined $C_m$. Analogous to the $\tilde G_i$-action also the $\tilde G$-action is $\it {CEF}$ on $M_{i-1}$ for $i=1,...,m$.
One therefore gets inclusions 
%\begin{equation}
% \begin{tikzcd}
%H^*_{\tilde G}(M) \arrow[r, "\gamma _1^*"]  & H^*_{\tilde G}(M_1) \arrow[r, "\gamma ^*_2"] & ... \arrow[r, "\gamma ^*_{m-1}"]  & H^*_{\tilde %G}(M_{{m-1}})\arrow[r, "\gamma ^*_m"]  &  H^*_{\tilde G}(C_m)
% \end{tikzcd}
%\end{equation}
\begin{equation}
H^*_{\tilde G}(M) \stackrel {\gamma _1^*} \longrightarrow H^*_{\tilde G}(M_1) \stackrel {\gamma _2^*} \longrightarrow ...  \stackrel {\gamma _{m-1}^*} \longrightarrow H^*_{\tilde G}(M_{m-1}) \stackrel {\gamma _m^*} \longrightarrow H^*_{\tilde G}(C_m).
\end{equation}
In order to calculate the maps $\gamma ^*_i$ above we also consider the Gysin map $\gamma _{i !}$ induced by the 
inclusion  $\gamma _i:M_i  \rightarrow M_{i -1}$. The composition $(\gamma _{i !})(\gamma ^*_{i})$ is given by the multiplication with the equivariant Euler
class (cf.~\cite{Kawakubo:1991} or ~\cite{AlldayPuppe:1993} for an account of equivariant Gysin homomorphisms, Euler classes, Thom classes etc.). In our case the Euler classes are given by the following Lemma.
\begin{lemma}
\label{EC}
For $i = 1,..., m-1$ the $\tilde G$-equivariant Euler class of the inclusion $\gamma _i:M_i  \rightarrow M_{i -1}$ is $~(t \otimes 1) \in H^*_{\tilde G}(M_{i -1}) \cong  \F_2[t] \otimes H^*_G(M_{i -1}) $.
\end{lemma}
\begin{proof}
The Morse function $g_i:M_{i-1} \rightarrow \R$ is $\tilde G$-equivariant, where the $\tilde G$-action on $\R$ is given by the trivial action of $G \subset \tilde G$, and the standard action of $\Z_2$ on $\R$. So the desired Euler class is just the pull back of the the Euler class of the inclusion $\{0\} \subset \R$, which is $~(t \otimes 1)$ as claimed.
\end{proof}

 Due to the above inclusions we view the elements  in $H^*_{\tilde G} (M_{ i})$ also as elements in 
$H^*_{\tilde G} (M^ {\tilde G}) = H^*_{\tilde G }(C_m)$ for $i = 0,...,m-1$.
In particular $s_j$ is divisible by $t^{m-1}$ in $H^*_{\tilde G} (M^ {\tilde G}) = H^*_{\tilde G }(C_m)$ (see ~\ref{FP}). We use the following notation
\smallskip
$s_J:= \prod _{j \in J} s_j$ and $\overline{s}_J:= \prod _{j \in J} \overline{s}_j$ for $J \subset \{1,...,r\}$ 
(cf.~Proposition~\ref{HM}). 
Note that $l(J):=f_i(p_{i,J}) = \sum_{\{j \in J \}} l_j ~ - ~\sum_{\{j \not\in J\}} l_j$ is independent of $i$. 
(We point out that $l(J)$
coincides with $L_J$ in ~\cite{FF}, but not with $ \ell (J) $ in ~\cite{Franz:2014}.)
 In the literature
%~\cite{FF},\cite{Franz:2014},\cite{Hausmann:book} 
the set $J$ is called \emph{short}, if $l(J) < 0$, and \emph{long}, if $l(J) > 0$.
For $i=1,...,m$ we define 
 $\mu _i (J): = \begin{cases}
i & \text{for $ ~l(J) > 0 $}\\ 0 & \text{for $ ~l(J) < 0$}
\end{cases}$\\
Let $\tilde R = H^*(\Z_2\times {(\Z_2) ^r}) = \F_2[t,t_1,...,t_r].$

\begin{theorem}
\label{HM_1}
Let $\ell $ be generic.
For $i = 0,..., m-1$ one has\\
\noindent
$H^*_{\tilde G }(M_{ {i}}) \cong 
\tilde R \langle s_J/t^{\mu _{i}(J)}; ~J \subset \{1,...,r\} \rangle$, i.e.~%
 the  $\tilde R$-subalgebra $H^*_{\tilde G}(M_{ {i}})$ of $H^*_{\tilde G}(M^{\tilde G})$ is generated by the elements
$\{s_J/t^{\mu_i (J)}; J \subset \{1,...,r\}\}$ as a free $\tilde R$-module.
\end{theorem}
 
We will give the proof of this theorem by alternating induction together with the proof of the following theorem.

\begin{theorem}
\label{HMG_1}
Let $\ell $ be generic.
For $i = 1,..., m
$ one has\\
\noindent
$H^*_{\tilde G _{i}}(M_{ {i}}) \cong 
\tilde R \langle s_J/t^{\mu _{i-1}(J)}; ~J \subset \{1,...,r\} \rangle /
 \tilde R \langle s_J/t^{i-1}(J),\overline s_J/t^{i-1}(J); l(J) > 0\rangle$.
%\end{equation}
\end{theorem}

\begin{proof}
The proof is by induction on $i$. The beginning, $i=0$, is clear for Theorem~\ref{HM_1} by Propositions~\ref{HM}. 
We proof Theorem~\ref{HMG_1} for $i$ under the assumption that Theorem~\ref{HM_1} holds for $i-1$.
%To get the beginning of our induction for Theorem(\ref{HMG_1} we have to calculate $H^_{\tilde G_1}(M_1)$.
We apply Theorem~\ref{MV} to the action of $\tilde G _i$ and the decomposition
$M_{i-1}=M^+_{i-1} \cup_{M^0_{i-1}} M^-_{i-1}$ with
$M^+_{i-1}:=\{w \in M_{i-1}; g_i(w) > 0\}, 
\linebreak  M^-_{i-1}:=\{w \in M_{i-1}; g_i (w) < 0\}$ 
and $M^0 _{i-1}:= \{w \in M_{i-1}; g_i (w) = 0\} = M_i.$
Since $l$ is generic, $g_i$ does not vanish on $C_i$.
So $M_i$ is a compact manifold since the set of critical points of $g_i$ is $C_i$; and $(M_i)^{\tilde G_i}$ is empty. 
Using the fact that $g_i$ is a perfect Morse function (see Proposition~\ref{perfect}), and comparing the total dimensions 
of the (non-equivariant) cohomology modules of the respective ambient spaces and their fixed point sets, one sees, by Smith theory, that not only $M_{i-1}$ but also ${M_{i-1}^\pm}$ are $\it {CEF}$ with 
respect to ${\tilde G_i}$. So the assumptions of Theorem~\ref{MV} are fulfilled.
This gives, with the notation corresponding to that in Theorem~\ref{MV},
\begin{equation}
H_{\tilde G_i}^{*}(M_i)
\cong {H_{\tilde G_i}^{*}(M)}/{({ker \xi^+}\oplus {ker \xi^-})}
\end{equation}
Since the $M_{i-1}^\pm$ are $\it {CEF}$, the kernels ${ker \xi^\pm}$ coincide with the kernels of the composition 
\begin{equation}
\begin{tikzcd}
H^*_{\tilde G_i}(M_{i-1}) \arrow{r}  & H^*_{\tilde G_1}(M^\pm_{i-1}) \arrow{r}   &  H^*_{\tilde G_i}(C_i \cap M_{i-1}^\pm)
\end{tikzcd}
\end{equation}
\noindent
and Proposition~\ref{FP} can therefore be used to identify these kernels. It follows from Proposition~\ref{FP} that under the map induced by  the 
inclusion $p_{i,J} \in M_{i-1}$ the element $s_I/t^{\mu _{i-1} (I)}$ is mapped to zero  if and only if $I \not \subset J$
and the element $\overline s_I/t^{\mu _{i-1} (I)}$ is mapped to zero  if and only if $I \cap J \neq \emptyset.$ Since we have that all $l_{j}$  are positive, 
%the function $l_i$ is increasing on the subsets of $\{1,...,r\}$ in the sense that $I \subset J$ implies $l_i(I) < %l_i(J)$.
$I \subset J$ implies $l(I) < l(J)$ .
One therefore gets that $ker \xi ^\pm$ are generated as free  $H^*_{\tilde G_i}(B{\tilde G_i})$-modules 
by  $\{\overline s_J/t^{i-1}; l(J) > 0\}$ and $\{s_J/t^{i-1} ; l(J) > 0\}$ 
respectively. 
Hence, by Theorem~\ref{MV} and Proposition~\ref{HM}, we have
\newline
\smallskip
{$H^*_{\tilde G _{i}}(M_{ {i}}) \cong 
\tilde R\langle s_J/t^{\mu_{i-1}(J)}; J \subset \{1,...,r\}\rangle /
  \tilde R \langle s_J/t^{i-1},\overline s_J/t^{i -1}; l(J) > 0\rangle $}.\\
\noindent	
So Theorem~\ref{HMG_1} holds for $i$.\\
 We next show that Theorem~\ref{HMG_1} for $i$, (and Theorem~\ref{HM_1} for $i-1$) imply Theorem~\ref{HM_1} for $i$ if $i < m$.
We consider the following diagram
\begin{equation}
\begin{tikzcd}
H^*_{\tilde G}(M_{ i-1}) \arrow{r}  
	  \arrow{d} 
	& H^*_{\tilde G}(M_{ i}) \arrow{d}\\
    H^*_{G}(M_{ i-1}) \arrow{r} 
 & H^*_{G}(M_{ i})\\
H^*_{\tilde G_i}(M_{ i-1}) \arrow{r}  
\arrow{u}   
 & H^*_{\tilde G_i}(M_{ i}) \arrow{u}
\end{tikzcd}
\end{equation}
The horizontal maps are induced by $\gamma_i:M_i \longrightarrow M_{i-1}$.
The vertical maps, which can be viewed as 'evaluation at $ t = 0$', are induced by the inclusions $G \subset \tilde G$ and $G \subset  \tilde G _i$, respectively. Since $H^*_{\tilde G}(M_{ i-1}), H^*_{\tilde G}(M_{i})$ and $H^*_{\tilde G_i}(M_{ i-1})$ are free modules over $\tilde R = \F_2[t,t_1,...,t_r]$, the `evaluation at $t=0$' is surjective, i.e.\\
$H^*_{\tilde G}(M_{i-1}) \stackrel{t=0} \longrightarrow H^*_{\tilde G}(M_{i-1})/t H^*_{\tilde G}(M_{i-1}) 
= H^*_{\tilde G}(M_{i-1}) \otimes _{\tilde R} R = H^*_{G}(M_{ i-1})$, etc. for these modules,
while $H^*_{\tilde G_i}(M_{i}) \stackrel {t=0} \longrightarrow H^*_{G}(M_{i})$ factors as \\
$H^*_{\tilde G_i}(M_{i}) \stackrel {t=0}\longrightarrow H^*_{\tilde G_i}(M_{i})/t H^*_{\tilde G _i}(M_{i}) 
\longrightarrow  H^*_G(M_i)$.
Note that 
$H^*_{\tilde G}(M_{ \mu})$ and $H^*_{\tilde G_i}(M_{ \mu})$ are isomorphic for $\mu < i$ since exchanging the coordinates $x_{i,j}$
and $x_{m,j}$ gives an equivariant homeomorphism of the $\tilde G_i$-space $M_{ \mu}$ and the $\tilde G$-space 
$M_{ \mu}$. But this does not hold for $\mu = i$.
In the above diagram elements of the form $(s_J/t^{i-1} ; l(J) > 0) \in H^*_{\tilde G}(M_{i-1})$ are mapped to zero as one goes 
to $H^*_{G}(M_i)$,
since the corresponding elements in $H^*_{\tilde G_i}(M_{i-1})$  go to zero by the above computation 
of $H^*_{\tilde G_i}(M_i)$.
This implies that $s_J/t^{i-1}$ can be divided by $t$ in $H^*_{\tilde G}(M_i)$. Since multiplication with 
$t$ is injective
in $H^*_{\tilde G}(M_i)$ we get a uniquely determined element $s_J/t^{i}$ for all $J$ 
with $l(J)>0$. On the other hand, from the above computation of $H^*_{\tilde G_i}(M_i)$  
one gets that $H^*_{\tilde G_i}(M_i)_{t=0} = H^*_{\tilde G_i}(M_i)/t H^*_{\tilde G_i}(M_i)$  is a free $R$-module generated 
by the elements $\{s_I ; l(I)<0\}$, since $s_J \equiv 0 \equiv \overline s_J$ in $H^*_{\tilde G}(M_i)/t H^*_{\tilde G}(M_i)$ for $l(J) > 0$.
(Note that $\overline s_J/t^{i-1} -s_J/t^{i-1} $  is divisible by $t$ in $H^*_{\tilde G_i}(M_i)$ , since $i<m$.) Hence the images of the elements $\{s_I ; l(I)<0\}$ from $H^*_{\tilde G}(M_i)$
are linearly independent  over $H^*(BG)=R$ in $H^*_{\tilde G_i}(M_i)_{t=0} \subset H^*_{G}(M_i)$. \\
%(Recall that from the definition of $\mu_i (J)$ we have $\mu_{i-1}(I) =\mu_i(I) = 0$  for $l (I)<0$, 
%and $\mu_{i-1}(J)+1 = \mu_i (J) = i$ for $l(J) > 0 .$)
We claim that the 
elements $\{s_I; l(I) < 0\} \sqcup \{s_J/t^i; l(J)> 0 \}$ (in other words $\{s_J/t^{\mu_i(K)};K \subset \{1,...,r\}\}$) freely generate $H^*_{\tilde G}(M_i)$ as $H^*(B\tilde G)$ -module.
Let $x \in H^*_{\tilde G}(M_i)$, then $\gamma_{i !}(x)$  can be written as
\begin{equation}
\gamma _{i!}(x) = \sum_{\{I;l(I)<0\}} \lambda_I s_I + \sum_{\{J;l(J)>0\}} 
\lambda_J s_J/t ^{i-1}
\end{equation} 
with $\lambda_I, \lambda_J \in \tilde R = H^*(B\tilde G)$.\\
So $tx = \gamma ^*_{i} \gamma _{i !}(x) = \sum \lambda_I 
(s_I)  + 
 %\sum \lambda_J 
%(s_J/t ^{i-1})  
\sum \lambda_J ts_J/t^{i}$ in $H^*_{\tilde G}(M_i).$\\
(Here we use the notation $\gamma ^*_i (s_K) = s_K$ as elements in $H^*_{\tilde G}(M^{\tilde G})$.)\\
Evaluating at $t = 0$ gives 
$\sum \lambda_I \gamma ^*_i (s_I) \equiv 0 ~(mod ~ t)$ in $H^*_G(M_i).$ \\
Since $\{\gamma ^*_i(s_I) \}$ are 
linearly independent in $H^*_G(M_i)$ over $R$, each coefficient $\lambda_I \in \tilde R$ must be 
divisible by $t$, i.e.~$\lambda_I = t \xi_I$ in $\tilde R$, and hence
$tx = \sum t \xi _I 
(s_I)  + 
\sum  t \lambda_J s_J/t ^i$. \\ 
Since division by $t$ is unique in $H^*_{\tilde G}(M_i)$, 
we get  $x = \sum \xi_I s_I  +  \sum \lambda_J s_J/t^i$.
%We use the notation $i^*_\nu(s_I) = s_I$ as elements in $H^*_{\tilde G}(M^{\tilde %G})$. So with this notation 
That means that
$\{s_J/t^{\mu_i (K)}; K \subset \{1,...,r\} \}$ generate the subalgebra $H^*_{\tilde G}(M_i)$  of $H^*_{\tilde G}(M^{\tilde G})$,
and these elements are linearly independent over $\tilde R$ = $H^*(B\tilde G)$, since this is clear after localization.
\end{proof} 
In a sense the (big) polygon spaces can be considered as a special case of the (big) chain spaces, if one allows the constant $c$ to also take the value $0$ in the definition of the chain spaces.
%\begin{remark}
%The equivariant cohomology of $N_0=M_{ m} = \{w \in M; f_i(w) = 0,$ for $i =1,...,m \}$ with respect to the $\tilde G$-action is %given by
%Theorem~\ref{HMG_1}. 
We have shown that the assumption ``$\ell $ generic'' implies, in particular, that $0 \in\R$ is a regular value of $g_m:M_{m-1} \longrightarrow \R$. Actually for any regular value $c \in \R$ of $g_m$ we can apply the last step of the proof of Theorem~\ref{HMG_1} to get the following generalization. Note that - due to the equivariant version of the Ehresmann fibration theorem -
the equivariant diffeomorphism type of $N_c$ does not change as $c$ moves in an intervall of only regular values of $g_m$ 
(cf. \cite{Franz:2014}).

\begin{theorem}
\label{HMG_2}
Let  $ \ell = (l_1,...,l_r)$ and  $\tilde \ell = (l_1,...,l_r,c)$ be generic. Then
\begin{equation}
H^*_{\tilde G }(N_c) \cong \\
\tilde R \langle s_K/t^{\mu _{m-1}(K)}; ~K \subset \{1,...,r\} \rangle / (S +\overline S)
\end{equation}
with $S:= \tilde R \langle s_I/t^{\mu_{m-1}(I)}; l(I) > - c \rangle,$ and $\overline S:= \tilde R \langle \overline s_J/t^{\mu_{m-1}(J)}; l(J) > c \rangle$. 							
\end{theorem} 
We finally are interested in the equivariant cohomology of $N_c$ with respect to the action of the subgroup $G \subset \tilde G.$
It can be obtained as the middle 
term in an short exact universal coefficient sequence.
\begin{proposition}
\label{HN}
The following sequence is exact and splits.

\begin{equation}
0 \longrightarrow {H_{\tilde G}^{*}(N_c)} \otimes _{\tilde R} R \longrightarrow {H_{G}^{*}(N_c)} 
\longrightarrow Tor^1_{\tilde R} ({H_{\tilde G}^{*}(N_c)}, R)  \longrightarrow  0 
\end{equation}
\end{proposition}  
\vspace{3 mm}
The above Proposition follows from the next Lemma, which is probably well known. Since we could not find a reference in the literature for the splitting of the short exact sequence in the case at hand, we will provide a proof here.
\begin{lemma}
\label{SP}
Let 
\begin{equation}
0 \longrightarrow A \stackrel{\alpha}{\longrightarrow} B
 \longrightarrow C 
\longrightarrow 0
\end{equation}
 be an exact sequence of free differential graded $R$-modules with 
$H_*(A)$ and $H_*(B)$ free over $R$. 
Then the exact sequence 
\begin{equation}                                                  
0 \longrightarrow coker ~\alpha _* {\longrightarrow} H_*(C)
 \longrightarrow ker ~\alpha _* 
\longrightarrow 0
\end{equation}
splits.
\end{lemma}
\begin{proof}
One has a short exact sequence
\begin{equation}                                                  
0 \longrightarrow Hom(C,D) {\longrightarrow} Hom(B,D)
 \longrightarrow Hom(A,D) 
\longrightarrow 0
\end{equation}
where $D:=coker ~\alpha_*$, and a corresponding long exact sequence 
\begin{equation}                                                  
... \rightarrow H^*(C;D) {\rightarrow} H^*(B;D)
 \rightarrow H^*(A;D) 
\rightarrow ...
\end{equation}
Since $H_*(A)$ and $H_*(B)$ are free, we have 
$H^*(A;D) \cong Hom(H_*(A),D)$ and \\ $H^*(B;D) \cong Hom(H_*(B),D)$. 
The map $H^*(C;D) {\rightarrow} H^*(B;D)$  is the composition of the surjection
$H^*(C;D) {\rightarrow} Hom(D,D)$  and the injection \\ $Hom(D,D) \rightarrow
Hom(H_*(B),D)$. 
The first map factors through $Hom(H_*(C),D)$. 
Hence $Hom(H_*(C),D) \rightarrow Hom(D,D)$ is also surjective and therefore\\ $D= coker ~\alpha _*
\rightarrow H_*(C)$ has a splitting.
\end{proof} 
\begin{proof}
Proof of Proposition ~\ref{HN}:
Let 
\begin{equation}                                                  
0 \rightarrow \tilde A \stackrel{\tilde {\alpha}}{\rightarrow} \tilde B
 \rightarrow \tilde C 
\rightarrow 0
\end{equation}
be a short exact sequence of free dg $\tilde R$-modules, which gives the sequence (3.18) in homology.
Then the sequence
\begin{equation}
0 \rightarrow  A \stackrel{ {\alpha}}{\rightarrow} B
 \rightarrow  C 
\rightarrow 0
\end{equation}
obtained from the above sequence by tensoring with $R$ over $\tilde R $, fulfills the 
hypothesis of Lemma 3.12 since $H_*(\tilde A) = H^*_{\tilde G}(M_{m-1})$ and $H_*(\tilde B) = H^*_{\tilde G}(M_{{m-1}})/S \oplus H^*_{\tilde G}(M_{{m-1}})/\overline S$ are free over $\tilde R$, and hence $H_*( A) = H^*_{ G}(M_{m-1})$ and $H_*( B) = H^*_{ G}(M_{{m-1}})/S \oplus H^*_{ G}(M_{{m-1}})/\overline S$ are free 
over $ R$. Therefore, by Lemma 3.12, one has a short exact sequence
%\begin{equation}                                                  
\[0 \rightarrow coker ~\alpha _* {\rightarrow} H^*_G(N_c)
\rightarrow ker ~\alpha _* 
\rightarrow 0
\]
%\end{equation} 
which splits; with $ coker ~\alpha _* \cong {H_{\tilde G}^{*}(N_c)} \otimes _{\tilde R} R$ and 
$ \ker ~\alpha _* \cong Tor^1_{\tilde R} ({H_{\tilde G}^{*}(N_c)}, R)$.
\end{proof}
For $i = m$ we can also express the result of Theorem~\ref{HMG_1} in form of the short exact sequences,
 which are free resolutions of the $\tilde R$-module $ H^*_{\tilde G}(N_c)$ (cf.(2.4) or (2.8))
\begin{equation}                                                  
0 \rightarrow H^*_{\tilde G}(M_{{m-1}}) \stackrel{\tilde {\alpha}_*}{\rightarrow} H^*_{\tilde G}(M_{{m-1}})/S 
\oplus H^*_{\tilde G}(M_{{m-1}})/\overline S
 \rightarrow H^*_{\tilde G}(N_c) 
\rightarrow 0
\end{equation}
or
\begin{equation}                                                  
0 \rightarrow \overline S \stackrel{\tilde 
	\iota}{\rightarrow} H^*_{\tilde G}(M_{{m-1}})/S 
 \rightarrow H^*_{\tilde G}(N_c) 
\rightarrow 0
\end{equation}
with $S
$ and $\overline S$ as above.\\
%$\overline S:=\{\overline s_J/t^{m-1}}; \ell(J) > 0\} $ and $S:= \{s_J/t^{m-1} ; l(J) > 0\}$ \\ 
%This is a free resolution of the $\tilde R$-module $H^*_{\tilde G}(N_c)$.
Therefore we can compute the tensor and tor term in (3.11) by just taking the cokernel and the kernel of the map $\tilde {\alpha}_*$ evaluated at $t=0$, i.e.~of $ \alpha := \tilde \alpha_* \otimes_{\tilde R} R$ or of $\iota := \tilde \iota \otimes_{\tilde R} R$.\\
We have the following exact sequence
\begin{equation}
0 \rightarrow ker ~\iota \rightarrow \overline S \otimes _{\tilde R}R  \stackrel \iota \rightarrow (H^*_{\tilde G}(M_{m-1})/S) \otimes_{\tilde R} R \rightarrow  coker ~\iota \rightarrow 0 
\end{equation}
where $\overline S \otimes _{\tilde R}R$ and $(H^*_{\tilde G}(M_{m-1} )/S) \otimes_{\tilde R} R$ are free $R$-modules of 
rank |$\{J; l(J) > c \}$|
and |$\{I;l(I) < -c \} $| = |$\{J; l(J) > c \}$| respectively.
We want to compute the map $\iota$. Since $\overline s_j =s_j + t^{m-1}t_j^n$ one has
$\overline s_J = s_j + \sum _{j\in J} t^{m-1}t_j^ns_{(J\setminus \{j\})} ~~+ ~t^{2(m-1)}(...)$,
so
\begin{equation}
 \tilde \iota (\overline s_J)  ~~~\equiv \sum _{\{j\in J; l({J\setminus \{j\}}) < -c\}} t^{m-1}t_j^ns_{(J\setminus \{j\})} ~~+ t^{2(m-1)}(...) 
 \end{equation}
 in $H^*_{\tilde G}(M_{m-1})/S$ for $l(J) > c > 0$, and 
\begin{equation}	
\iota (\overline s_J/t^{m-1}) ~~~\equiv \sum _{\{j\in J; l({J\setminus \{j\}}) < -c\}} t_j^ns_{(J\setminus \{j\})} 
\end{equation}
in $(H^*_{\tilde G}(M_{m-1})/S) \otimes _{\tilde R} R$. 	
\begin{corollary}
\label{free?}
For $n > 0$ the equivariant cohomology $H^*_G(N_c)$ is a free $R$-module if $c \geq l_j$ for $j=1...,r$, i.e.  if
$ (l_1,...,l_r,c)$  is a dominated length vector in the sense of ~\cite{FHS}.
On the other hand  $H^*_G(N_0)$ is never free (cf.~\cite{Franz:2014}), Lemma 4.3).
\end{corollary}	
\begin{proof}
We use the criterion, established by Smith theory, that  $H^*_G(N_c)$ is free if and only if
 $dim_{\F_2} H^*(N_c) = dim_{\F_2} H^*(N_c^G)$. Assume that $c=0$.
Since $n>0$ we can rename the variables $y_{n,1},...,y_{n,r}$ as $x_{m+1,1},...,x_{m+1,r}$, thereby replacing $m$ by $m+1$, and $n$ by $n-1$. We can now apply Cor.~\ref{perfect} in the new setting to obtain
that $dim_{\F_2}H^*(M_m) = 2^r$ (cf.~\cite{Franz:2014}, Prop.~3.3).
On the other hand one can also use Theorem~\ref{HM_1}, Theorem~\ref{HMG_1} and Proposition~\ref{HN} to compute $H^*_{\Z_2}(N_0^G)$ and 
$H^*(N_0^G)$. This corresponds to the case $n=0$. One obtains from the sequence (3.20) that $dim_{\F_2}H^*(N_0^G) = 2^r - 2rk(\iota) < 2^r$, because there always exists a $J$ with $l(J) > 0$
such that $ l(J \setminus \{j\}) < 0$ for some $j \in J$.  Therefore $\iota $ is not trivial. Hence  $H^*_G(N_0)$ is never free.
Also for $c \neq 0$ one can use the above calculation to get $dim_{\F_2}H^*(N_c^G) = 2|\{J; l(J) > c \}| - 2rk(\iota)$, but $\iota $ is trivial here if  $c\geq l_j$ for $j=1,...r$. 
 To calculate $H^*(N_c)$ we rename the variables
$(y_{\nu,1},...,y_{\nu,r})$ as $x_{m+1,1},...,x_{m+1,r}$ for $\nu = 1,..,n$ thus replacing $m$ by $m+n$ and $n$ by $0$. The group $G$ in this new setting 
is just $\{1\}$  and $\tilde G = \Z_2$.
So $\overline s_j  = s_j + t^{m+n-1}$ in this setting. But the number of hyperplanes intersecting the product of spheres is $m$. And therefore the appropriately
modified Remark 3.10 gives 
 \begin{equation}
 H^*_{\Z_2 }(N_c) \cong \\
 \tilde R\langle s_K/t^{\mu _{m-1}(K)}; ~K \subset \{1,...,r\} \rangle / (S +\overline S)
 \end{equation}
 with $S:= \tilde R \langle s_I/t^{\mu_{m-1}(I)}; l(I) > - c\rangle ,$ and $\overline S:= \tilde R \langle \overline s_J/t^{\mu_{m-1}(J)}; l(J) > c\rangle $.
 But $ \iota $ turns out to be trivial in this case since
 $\overline s_J/t^
 {m-1} - s_J/t^{m-1} \equiv 0 ~mod ~t$. So $dim_{\F_2} H^*(N_c) = dim_{\F_2}  ker~\iota + dim_{\F_2}  coker~\iota  =  2|\{J; l(J) > c \}|$ (see (3.20)).
 \end{proof}

Examples below show (see Example~\ref{bcs} (3)), that  $H^*_G(N_c)$ can be free for $c \neq 0$  even if $ (l_1,...,l_r,c)$  is not a dominated length vector.
Due to the equivariant version of the Ehresmann fibration theorem varying the constant $c$ between two adjacent critical values of $g_m$ does not change the equivariant  diffeomorphism type of $N_c$. So $N_c \cong N_0$ if $ 0 < c < cr_{min}$ where
$cr_{min}$ denotes the minimal positive critical value of $g_m$.
We extend the analogue of ~\cite{Franz:2014}, Cor.6.4 to the situation where $c$ is not necessarily equal to $0$. Two lenght vectors are considered equivalent, if they induce the same notion of of 'long' and 'short' index sets.

\begin{proposition}
\label{maximal}

(1) If $r = 2k+1$ then $H^*_G(N_c)$ has syzygy order k if an only if  $ (l_1,...,l_r)$  is equivalent to $(1,...,1)$ and $0 < c < cr_{min}$.

(2) If $r = (2k+2)$ then  $H^*_G(N_c)$ has syzygy order k if an only if  $ (l_1,...,l_r)$  is equivalent to $(0,1,...,1)$ and $0 < c < cr_{min}$.

\end{proposition}

\begin{proof}
The proof is a modification of the proof of ~\cite{Franz:2014}, Prop.6.4.

Let $ \mathcal  L_c $: = $\{J; l(J) > c \} $
and $ \mathcal S_c $: = $\{I; l(I) < -c\}.$ The map which assigns to a subset $J \subset \{1,...,r\}$ its complement $\overline J$ gives a bijection between  $ \mathcal  L_c $  and $ \mathcal  S_c $ . Let $\mathcal L'_c$  be the subset
 of  $ \mathcal  L_c $ = $\{J; l(J) > c \} $,
such that there exits an $j \in J$ with $(J \backslash \{j\}) \in    \mathcal  S_c $: = $\{J; l(J) < - c \} $.\\
(1) Assume $r=2k+1$ and that the syzygy order of  $H^*_G(N_c)$  is k.  Since $H^*_G(N_c)$ is not free,  $ \mathcal  L'_c $ can not be empty;
and since the syzygy order is $k$,  any index set $J \in   \mathcal  L'_c $ contains at least $k+1$ indices $j$, such that  $(J \backslash \{j\}) \in    \mathcal  S_c $   (cf.~\cite{Franz:2014}, Prop.6.3). In particular  any  set $J \in   \mathcal  L'_c $ must have at least $k+1$ elements.  The complement $\overline J$  of $J$ has at most m elements and lies in  $ \mathcal  S_c $ . But $\overline J  \cup \{j\}$ is in  $ \mathcal  L'_c $, since its complement
 $J \backslash \{j\}$ is in  $ \mathcal  S_c $. So $\overline J  \cup \{j\}$ and also $J$  must have precisely $k+1$ elements for a sets $J \in \mathcal  L'_c. $  Removing an element $j$ from a set  $J$ as above and replacing it by an element $i$  from $\overline J$ gives again a set in  $ \mathcal  L'_c .$ 
It therefore follows, that the sets in  $ \mathcal  L'_c $  are all those having precisely $k+1$ elements and the
 sets in $ \mathcal  L_c $ all those which have at least $k+1$ elements.This means that sets with less than $k+1$ elements are in  $ \mathcal S_c $.
All together  $\mathcal  L_c $,  resp. $ \mathcal  S_c $, coincide with the  long, resp. short, subsets  for the length vector $(1,...,1)$,
and $c < cr_{min}$   since there must be an index set $K$ with $l(K) = cr_{min}.$ This proves part (1).\\
(2) We assume again that the syzygy order of $H^*_G(N_c)$ is $k$, but this time $r = 2k+2$. We define  $ \mathcal  L_c $  and  $ \mathcal  S_c $  and  $ \mathcal  L'_c $  as before. We may assume that $ l_1 < l _2  \leq l_3 ...\leq l_r$   because the equivariant diffeomorphism type of $N_c$ does not change under small enough perturbations of the $(l_1,...l_r)$ and of $c$. Again $  \mathcal  L'_c $ is not empty, and for any set $J \in   \mathcal  L'_c $ , one has $J \backslash \{j\} \in   \mathcal  S_c $  for at least $k+1$ elements in $J$. Arguing similarly to case (1) one gets that either both $J$ and $\overline J$ have precisely $k+1$ elements or $J$ has $ k+2$ elements and therefore $\overline J$ has $k$ elements. If $J \in  \mathcal  L'_c $  then  $\overline {J \backslash \{j\}} = \overline J \cup j$ is also in  $ \mathcal  L'_c $ for at least $k+1$ elements  $j \in J$. So  $ \mathcal  L'_c $ must contain index sets with $k+2$ elements and also index sets with $k+1$ elements. 
Assume $J \in   \mathcal  L'_c $  contains $1$ and has precisely $k+1$ elements. Then  $\overline {J \backslash \{j\}} = \overline J \cup j$ is also in  $ \mathcal  L'_c $ for all elements  $j \in J$. Also  $I:=\overline {J \backslash \{j\}} = \overline J \cup j$ contains at least $m+1$ elements  
$i$ , such that  $\overline {I \backslash \{i\}} = J \backslash \{j\} \cup \{i\}$ is also in  $ \mathcal  L'_c $. Listing the indices  occuring
 in $J$, resp. $\overline J$ weakly  increasing one gets two sequences $1, a_2,...a_{k+1}$ resp. $b_1,...,b_{k+1}$. The above argument shows that one can replace $a_\nu$ by $b_\nu $ if  $a_\nu > b_\nu$ for $\nu = 2,...,{m+1}.$  For the index sets, $K$ 
and $\overline K$, obtained this way one still has $1 \in K$ and $K \in \mathcal L'.$  But $l(K) < l(\overline K)$, which is impossible,
 since $l(K) > c.$  So there can't be index sets $J \in \mathcal L'_c$ containing $1$ and precisely $ {k+1}$ elements.
Assume $J \in   \mathcal  L'_c $  contains $1$ and has $k+2$ elements. Similar to the above reasoning one sees that replacing an index $j \neq 1$ in $J$  by an index $i$ in $\overline J$ gives again an index set in  $ \mathcal  L'_c .$ 
It follows that all index sets, which contain $1$ and have precisely $k+2$ elements are in  $ \mathcal  L'_c $. So   $\mathcal  L_c $,  resp. $ \mathcal  S_c $, coincide with the  long, resp. short, subsets  for the length vector $(0,1,...,1)$.
As above one sees that $c < cr_{min}.$ 
\end{proof}
\begin{remark}
\label{n=0}
(1) We would like to point out that Theorem~\ref{HM_1} and Theorem~\ref{HMG_1} give complete information about the product structure
of the equivariant cohomology with respect to the $\tilde G _i$-actions, but Proposition~\ref{HN} gives only partial information about the product in $H^*_G(N_c)$. \\
(2) Theorem~\ref{HM_1}, Theorem~\ref{HMG_1} and Proposition~\ref{HN} can be applied if $n = 0$
and $G = \{1\}$ (see e.g. the proof of Cor.~\ref{free?}). As a special case one obtains the $\Z_2$-equivariant and the non-equivariant cohomology of the spaces in
Remark~\ref{spaces},(2) -(4), which are studied in several papers, see e,g, ~\cite{FS},~\cite{FHS}, ~\cite{FF}, ~\cite{Hausmann:book}. In case $m>2, n=0$ and $c=0$ the term ${H_{\tilde G}^{*}(N_0)} \otimes _{\tilde R} R = {H_{\Z_2}^{*}(E_m(\ell))} \otimes _{\F_2 [t]} \F_2$, including the multiplicative structure, is 
just $H^{(m-1)*} (E_m(\ell);\F_2)$ studied in detail in ~\cite{FF}, Section 4.
One gets that
$dim_{\F_2} H^*(E_m(\ell))$ is smaller than $2^r$. A sharp upper bound is contained in ~\cite{FS}, Thm. 2].
\\
(3) Since, for $i = 1,...,m-1$, the spaces $M_{ i}$ are $\it{CEF}$ with respect to the $\tilde G$-action,
one obtains the equivariant cohomology with respect to subgroups $G' \subset \tilde G$, and in particular the 
non-equivariant cohomology (for the trivial subgroup $\{1\}$), just as the tensor product 
$H^*_{\tilde G}(M_{i}) \otimes _{H^*(B \tilde G)} H^*(BG').$
\end{remark}
While $H^*_{G'}(M_i)$ is free over $H^*(BG')$ for any $G' \subset \tilde G _i$ and $i=1,...,m-1$, the equivariant cohomology $H^*_{\tilde G}(M_m)$ is torsion, which already follows from the fact that $(M_m)^{\tilde G} = \emptyset$.
As mentioned above the equivariant cohomology $H^*_G (M_m)$, $(M_m=N_0)$, with respect to $G$ is not free, but it is often torsion-free over $R$ 
(see ~\cite{Franz:2014}, Section 5 and 6). 
We will not perform explicit calculations here for the general case. For the ``big polygon spaces'' (and the corresponding 
${(S^1)}^r$-action in the complex situation) these are done and discussed in ~\cite{Franz:2014}. This could be imitated in the real case at hand in a similar vein. 
But the following example (cf. ~\cite{Franz:2014}, Prop.~5.1) gives - from the view point of syzygies - perhaps the most interesting special 
%, namely the equilateral,
 case and already shows some typical features of the general case for the "big polygon spaces". Later on we discuss some examples of "big chain spaces".
\begin{example}
\label{SYZ}
We assume $m=2, n=1,
r = 2k + 1, k \geq 0$ and $\ell:= (1,...,1)$.
Under this assumptions one has $l(J) > 0$ if and only if $|J| > k$.
The map $\iota$ defined above turns out in this case to be trivial on $\overline s_J/t$ for $l(J) > k+1$. For $l(J) = k+1$ it coincides with the following boundary map
in the Koszul complex $\delta:\Lambda^{k+1}_R(\sigma_1,...,\sigma_r) \rightarrow \Lambda^k_R(\sigma_1,...,\sigma_r)$
if one puts $R \langle \overline s_J/t; l(J) = k+1 \rangle \cong \Lambda^{k+1}_R(\sigma_1,...,\sigma_r)$
and $R \langle s_I; l(I) = k \rangle \cong \Lambda^{k}_R(\sigma_1,...,\sigma_r)$ .
We therefore get
\begin{equation}
%H^*_{\tilde G} \otimes_{\tilde R} R 
coker ~\iota \cong 
R \langle s_I, |I| < k \rangle \oplus 
coker~(\delta:\Lambda^{k+1}_R(\sigma _1,...,\sigma _r) \rightarrow \Lambda^k_R(\sigma_1,...,\sigma_r)) 
\end{equation}
and
\begin{equation}
%Tor_1^{\tilde R}(H^*_{\tilde G}, R)
ker ~\iota \cong 
R\langle s_J/t, |J| > k+1\rangle \oplus 
ker ~(\delta:\Lambda^{k+1}_R(\sigma _1,...,\sigma _r) \rightarrow \Lambda^k_R(\sigma _1,...,\sigma _r))
\end{equation}
%Here $R\{s_I, |I| < k+1\}$  denotes the free $R$-module generated by $\{s_I, |I| < k+1\}$ , etc.
Finally, from the short exact Künneth sequence, which splits, we get
$H^*_G(N_0) \cong coker ~\iota \oplus ker ~\iota $ as $R$-modules.
In particular, since $coker~(\delta:\Lambda_R^{k+1} \rightarrow \Lambda_R^{k})$
is a k-th, but not a $(k+1)$-th syzygy  and  $ker~(\delta:\Lambda_R^{k+1} \rightarrow \Lambda_R^{k})$ is a (k+2)-th syzygy, 
we get the following result. 
\end{example}
For Example ~\ref{SYZ} one has:
\begin{corollary}
\label{SY}
The equivariant cohomology $H^*_G (N_0)$ is a k-th syzygy, but not a $(k+1)$-th
syzygy.
\end{corollary} 
\begin{remark}
\label{SO}
The above Corollary~\ref{SY} shows that the maximal bound (namely $k$) for the syzygy order (in the non-free case) given by 
Proposition~\ref{thm:torsion-free-hd-1} for an action of $(\Z/2)^{2k+1}$ can be realized by the equivariant cohomology of a compact manifold. It is pointed out in \cite{Franz:2014},(5.2) that, using such ``maximal'' examples, one can easily realize all other orders of syzygies allowed
by Proposition~\ref{thm:torsion-free-hd-1}. One can just extend the action to a larger rank torus letting the extra coordinates act trivially.
This obviously changes the rank of the torus, but it does not change the syzygy order of the equivariant cohomology. 
It corresponds to extending the length vector by another coordinate equal to $0$, and it it is easy to check that this does not  change the syzygy order but increases the rank of the torus acting. 
In \cite{Franz:2014}
there is a careful discussion of the effect of different length vectors on the syzygy order in case of ${(S^1)}^r$-actions.
This could as well be imitated for the $(\Z_2)^r$-manifolds considered here. 
\end{remark}

We finish with a few examples of "big chain spaces", a class of spaces which is not considered in  \cite{Franz:2014}. 
They show that the syzygy order of $H^*_G(N_c)$ depends in a rather delicate way on the length vector $\ell$ and the constant $c$. But using 
Theorem ~\ref{HMG_2} and Proposition ~\ref{HN} the calculation of the equivariant cohomology is straight foreword and we leave the details to the reader. Again one can change the rank of the torus acting without changing the syzygy order by adding coordinates  equal to $0$ to the length vector. In particular one obtains corresponding examples for tori of even rank this way.

\begin{example}
\label{bcs}
(1) Let $r=2k+1$ and  $ \ell = (1,...,1).$   The critical values of $g_m$ are  $\{ -r, -(r-2),...,-1,1,...,(r-2), r\}$. Recall that for a regular value, $c$ of $g_m$, one has $H^*_G(N_c) \cong H^*_G(N_{-c}.$  Let $0 \leq c$ be a regular value of $g_m$, then\\

 $H^*_G(N_c )   \left \{
 \begin{array}{ccc}
$has syzygy order k$
 &  $if$ & 0 \leq c <1\\
$is free$  &  $ if $ & c >1 \\
\end{array}\right.$\\

\noindent This does not mean that $ H^*_G(N_c)$ is the same for all $c > 1$ . The rank decreases as $c$ increases, crossing  critical values of $g_m$; in particular $H^*_G(N_c)  =  0 $ if $ c > r$.\\
%$ is an $k$-th syzygy , if $ 0 <  c <  1$, and it is free for $c$ > 1 (see above).

\noindent (2) Let $r=2k+1$,  $\ell = (2,2,3,...,3)$ and $c$ a regular value of $g_m$, then \\

$H^*_G(N_c )  \left \{
 \begin{array}{ccc}
$has syzygy order k$
 &  $if$ & 0 \leq c <1\\
$ has syzygy order (k-1)$ & $if$ & 1 < c < 3\\
$is free of decreasing rank $  &  $ if $ & c >3 \\
\end{array}\right.$\\

\noindent (3) Let $\ell = (2,2,2,3)$, then \\

$H^*_G(N_c)   \left\{
\begin{array}{ccc}
$has syzygy order 0$ & $if$ & 0 \leq c < 1\\
$ is free$ & $if$ & c > 1\\
\end{array} \right.$\\
\end{example}

\noindent Note that for the first two cases in Example \ref{bcs} one gets the same result for $c = 0,$ i.e.~for the big polygon spaces, but not for all values of $c,$ i.e.~not for all big chain spaces.
In case (3) one gets freeness of the equivariant cohomology even for constants $c$ which do not 
dominate $\ell$ (cf. Cor.~\ref{free?}).

\end{document}